\documentclass[12pt]{amsart}
\usepackage{latexsym}
\usepackage{amsthm}
\usepackage{amsmath}
\usepackage[dvips]{graphics}
\usepackage{graphicx}
\usepackage{amssymb}
\usepackage{color}
\usepackage{setspace}
\usepackage{ulem}
\usepackage{epsfig}
\usepackage[all]{xy}

\theoremstyle{plain}
\newtheorem{theorem}{Theorem}
\newtheorem{lemma}{Lemma}

\theoremstyle{definition}

\newtheorem*{theorem*}{Theorem}
\newtheorem*{lemma*}{Lemma}
\newtheorem*{claim*}{Claim}
\newtheorem*{exercise*}{Exercise}
\newtheorem*{note*}{Note}
\newtheorem*{example*}{Example}
\newtheorem*{problem*}{Problem}
\newtheorem*{solution*}{Solution}
\newtheorem*{remark*}{Remark}

\newtheorem{corollary}[]{Corollary}
\newtheorem{example}[]{Example}
\newtheorem{definition}[]{Definition}

\newcommand{\x}{\text{\boldmath{$x$}}}
\newcommand{\g}{\text{\boldmath{$g$}}}
\newcommand{\w}{\text{\boldmath{$w$}}}
\newcommand{\vv}{\text{\boldmath{$v$}}}

\newcommand{\wout}{\backslash}                          

\newcommand{\integer}{\ensuremath{\mathbb{Z}}}  

\newcommand{\cross}{\ensuremath{^{\times}}}     

\newcommand{\iso}{\cong}                        

\newcommand{\ssum}[2] {\overset{#2}{\underset{#1}{\sum}}}  
\newcommand{\sprod}[2] {\overset{#2}{\underset{#1}{\prod}}} 
\newcommand{\s}[3] {\overset{#3}{\underset{#2}{#1}}}

\begin{document}

\title{}
\begin{center}
\textbf{ Lifting Automorphisms of Quotients by Central Subgroups }

Ben Kane, Andrew Shallue
\vspace{.25in}

Department of Mathematics\\ University of Wisconsin-Madison \\
  480 Lincoln Dr\\  Madison, WI 53706, USA\\
kane@math.wisc.edu\footnote{now bkane@math.uni-koeln.de}, shallue@math.wisc.edu\footnote{now ashallue@iwu.edu}
\end{center}
\date{\today}
\begin{abstract}
Given a finitely presented group $G$, we wish to explore the conditions under 
which automorphisms of quotients $G/N$ can be lifted to automorphisms of $G$.  
We discover that in the case where $N$ is a central subgroup of $G$, the 
question of lifting can be reduced to solving a certain matrix equation.  We 
then use the techniques developed to show that $Inn(G)$ is not characteristic 
in $Aut(G)$, where $G$ is a metacyclic group of order $p^n$, $p\neq 2$.

\end{abstract}
\maketitle

\section{Introduction}

Let $F$ be the free group on $n$ letters, and let $G$ be a quotient of that 
group.  We will be working with a given presentation of $G$, namely
$$
G := \left<x_1, x_2, \dots x_n | r_1(\x), r_2(\x), \dots r_m(\x)\right>,
$$
where $\x$ represents the $n$-tuple $(x_1,x_2,\dots, x_n)$.  This vector 
notation continues throughout the paper.  For later ease of exposition, we 
will think of the relations $r_k$ as noncommutative monomials on $n$ variables 
$(x_1,\dots, x_n)$ defined by 
$$
r_k(\x)=\sprod{l=1}{s_k}{x_{j_{k,l}}^{e_{k,l}}}
$$
with $e_{k,l}\in \{\pm 1\}$.

It is a well known fact that if $N$ is a characteristic subgroup of $G$, then 
automorphisms of $G$ induce automorphisms of $G/N$ canonically by acting on 
the coset representatives.  However, much less is known about the conditions under which a lift of an element of $Aut(G/N)$ exists.

Study so far has focused on the case where $G$ is the free group $F$.  Many 
techniques have been developed to show whether automorphisms of various 
quotients of $F$ are tame (i.e. for which lifts to $F$ exist).  See for 
instance \cite{FreeAut65,MetaB65,NonComm97, TameFG98}.

We show in the case where $N$ is a central subgroup of an arbitrary group $G$ 
that automorphisms of $G/N$ are in one-to-one correspondence with solutions to a 
certain set of matrix equations that depend only on the relations.

\section*{Acknowledgements} 
The authors are grateful to N. Boston for giving motivation to the problem and for valuable conversation.

\section{Homomorphic Lifts}
First, assume that $G$ is a finitely presentable group, and $N$ is a 
cyclic, central subgroup of $G$, generated by an element $z$.  We will later 
generalize to the case where $N$ is not cyclic.  We wish to study when 
automorphisms of $G/N$ can be lifted to automorphisms of $G$.  We first give a 
condition for when such automorphisms lift to homomorphisms of $G$, so in this section lift means homomorphic lift.

We will also consider $G$ as the image of $F$ under the canonical quotient map 
$\pi: F\to F/R$, where $R$ is the normal closure of the set 
$\{r_1(\x) \dots r_m(\x)\}$.  However, this will be for convenience only.  In 
practice, working with the relations will suffice, as evidenced by the 
following lemma, the proof of which is immediate.

\begin{lemma}\label{normclosure}
Let $\theta \in Hom(F,H)$ be given.  Then $\theta(r_k(\x))=1$ for all $k$ if and only if
$\theta(R)=1$.
\end{lemma}

The following definition will help to clarify our direction in this paper:
\begin{definition}\label{extends}
We say that an $n$-tuple $\g=(g_1,g_2,\dots, g_n)\in G^n$ \emph{extends} to a 
homomorphism if the homomorphism $\theta:F\mapsto G$, defined by 
$\theta(x_i)=g_i$, factors through $R$.  In this case $\theta\circ \pi^{-1}$ 
is well defined and defines an element $\psi\in Hom(G,G)$.
\end{definition}

\begin{lemma}\label{elminduce}
An $n$-tuple $\g\in G^n$ extends to a homomorphism if and only if $r_k(\g)=1$ for every 
$k= 1,\dots,m$.
\end{lemma}
\begin{proof}
Let $\g\in G^n$ be given.  Consider $\theta$ defined as above.  Note that $\g$ 
extends to a homomorphism if and only if $\theta(R)=1$.  By lemma \ref{normclosure}, 
$\theta(R)=1$ if and only if $\theta(r_k(\x))=1$ for all $k$.  So it remains to show that 
$\theta(r_k(\x))=1$ if and only if $r_k(\g)=1$.  However, since $\theta$ is a 
homomorphism, 
$$
\theta(r_k(\x))=r_k(\theta(\x))=r_k(\g).
$$
\end{proof}

The following definition is vital, since homomorphic lifts are the focus of 
this paper.
\begin{definition}
A \emph{lift} of $\varphi\in End(G/N)$ is $\psi\in End(G)$ such that
$$\psi(g)N=\varphi(gN)\qquad \text{for every }g\in G.$$
\end{definition}

\begin{theorem}\label{reln}
If $\varphi \in End(G/N)$, then there is a one-to-one correspondence 
between lifts of $\varphi$ to $End(G)$ and $n$-tuples $\g\in G^n$ such that 
$g_i N =\varphi(x_iN)$ and $r_k(\g)=1$ for every $k\in 1,\dots,m$.
\end{theorem}
\begin{proof}
First suppose $\psi$ is a lift of $\varphi$.  Then by definition $\varphi(x_iN) = \psi(x_i)N$, so we set $g_i := \psi(x_i)$.  Also, 
$$
1 = \psi(1) = \psi(r_k(\x)) = r_k(\psi(\x))=r_k(\g)
$$
since $\psi$ is a homomorphism.

Conversely, suppose we have an $n$-tuple $\g$ such that $g_iN = \varphi(x_iN)$ 
and $r_k(\g) = 1$ for every $k = 1 \dots m$.  Then by Lemma \ref{elminduce}, 
$\g$ extends to $\psi \in End(G)$, where $\psi(x_i) = g_i$.  Thus
$$
\psi(x_i)N = g_iN = \varphi(x_iN)
$$
and this is exactly the definition of $\psi$ being a lift of $\varphi$.
\end{proof}

We now give a nice characterization of the lifts of $\varphi$.  For this we define a certain matrix and a vector.

Define $m_{ij}$ to be the degree of $x_j$ in the commutative image of the word 
$r_i(\x)$.  Note that since $r_i(\x)=\sprod{l=1}{s_i} x_{j_{i,l}}^{e_{i,l}}$,
$$
m_{ij}=\ssum{\s{l\in 1..s_i}{j_{i,l}=j}{}}{}e_{i,l}.
$$
We consider the matrix $M:=\left(m_{ij}\right)$.
To make the construction of $M$ clear, we give an example.
\begin{example}
For 
$$
r_1(\x)=x_1^2\cdot x_2^{-1}\cdot x_1^{-5}\cdot x_2^{-1},\qquad r_2(\x)=x_1\cdot x_2^{-3}\cdot x_1^7,
$$
we have
$$M=
\left( 
\begin{smallmatrix}
-3& -2\\
8& -3
\end{smallmatrix}
\right)
$$
\end{example}

For $\varphi\in Aut(G/N)$, fix a set of coset representatives 
$\overline{x_i}\in \varphi(x_iN)$.  Since 
$$
N=\varphi(r_k(\x)N)=r_k(\varphi(\x N))=r_k(\overline{\x})N,
$$ 
it is clear that $r_k(\overline{\x})\in N$.  Since $N$ is generated by $z$, we
can choose $w_i$ such that $r_i(\overline{\x})=z^{-w_i}$.  Note that 
$\w:=(w_1,\dots,w_m)$ is only defined up to the order of $N$.

\begin{theorem}\label{matrix}
The lifts of $\varphi$ are in one-to-one correspondence with solutions 
$\vv=(v_1,\dots v_n)$ to the matrix equation 

$$
M\vv=\w \pmod{\#N}
$$
where, if $N$ is infinite, we simply mean the matrix equation on the integers.
\end{theorem}
\begin{proof}
We know from above that lifts are in one-to-one correspondence with $\g\in G^n$ such 
that $r_k(\g)=1$ and $g_iN=\varphi(x_iN)$.

However, if $g_iN=\varphi(x_iN)=\overline{x_i}N$, then $g_i\in\overline{x_i}N$.
But then $g_i=\overline{x_i}z^{v_i}$ for some $i$.  So $g_iN=\varphi(x_iN)$ 
if and only if $g_i=\overline{x_i}z^{v_i}$ for some $v_i$.  So lifts are in 
one-to-one correspondence with $\g\in G^n$ such that $r_k(\g)=1$ and 
$g_i=\overline{x_i}z^{v_i}$.  

Since $z$ is central,
$$
\begin{array}{lcl}
r_i(\overline{x_1}z^{v_1},\dots,\overline{x_n}z^{v_n}) &=& r_i(\overline{\x})r_i(z^{v_1},z^{v_2},\dots,z^{v_n})\\
&=& z^{-w_i}z^{\ssum{j=1}{n} m_{ij}v_j}\\
&=& z^{-w_i+\ssum{j=1}{n} m_{ij}v_j}
\end{array}
$$

But this is equal to $1$ if and only if 
$-w_i+\ssum{j=1}{n} m_{ij}v_j=0\pmod{\#N}$.  This corresponds exactly to a 
solution of the above matrix equation.

\end{proof}
Having shown the result for $N$ cyclic, it is easy to generalize to the case 
when $N$ is not cyclic.

If $N$ is a central subgroup of $G$, generated by $z_1,\dots, z_t$, then 
we have 
$$r_k(\overline{\x})=z_1^{-w_{1,k}}z_2^{-w_{2,k}}\cdots z_t^{-w_{t,k}}$$

\begin{corollary}
The lifts of $\varphi$ are in one-to-one correspondence with solutions of the 
matrix equation 

$$
\left(
\begin{smallmatrix}
M&      &\\
 &\ddots&\\
 &      & M\\
\end{smallmatrix}
\right)\vv = 
\left( 
\begin{smallmatrix}
\w_1\\
\vdots\\
\w_t
\end{smallmatrix}
\right)
\pmod{\#N},
$$
where, if $N$ is infinite, we simply mean the matrix equation on the integers.

\end{corollary}
\begin{proof}
The proof follows from the proof of the previous theorem, noting that each 
generator of $N$ commutes.
\end{proof}

\section{Automorphic Lifts}

In this section we investigate when such homomorphic lifts are 
automorphic.  As before, we assume that $G$ is finitely presented and $N$ is a 
central subgroup of $G$.

\begin{lemma}
If $N$ is abelian, finitely generated, and $\psi\in End(N)$, then $\psi$ 
surjective implies $\psi$ injective.
\end{lemma}

This lemma follows from the fundamental theorem of abelian groups and the fact 
that the rank of the image plus the rank of the kernel equals the rank of $N$.

\begin{lemma}\label{homtoaut}
A lift $\psi\in End(G)$ of $\varphi\in Aut(G/N)$ is an automorphism if and 
only if $\psi(N)=N$.
\end{lemma}
\begin{proof}
Consider $K:=Ker(\psi)$ and $H:=Im(\psi)$.  

Since $\psi$ is a lift of $\varphi$, we have the identity
$$N=\psi(K)N=\varphi(KN).$$
As $\varphi$ is injective, it follows that $KN=N$, and hence $K\subseteq N$.  
So $K=Ker(\psi|_N)$.  Therefore $\psi$ is injective on $G$ if and only if it 
is injective when resticted to $N$.  Because $N$ is abelian and finitely 
generated by assumption, it will suffice to show $\psi|_N$ is surjective, 
even if $N$ is infinite.

Moreover, since $\psi(G)N=\varphi(GN)=G$, it follows that $HN=G$.  
If $\psi(N)=N$, then $N\subseteq H$, so that $G=HN=H$.  Conversely, $N$ is 
in the image of $\psi$, and since for $g\in G$, $\psi(g)N=\varphi(gN)$, and 
$\varphi(gN)=N$ if and only if $g\in N$, we know that the preimage of $N$ is $N$.  
So we have $\psi$ surjective if and only if $\psi(N)=N$.  

\end{proof}

In the case where $\# N$ is finite and squarefree, the following result will 
show that the previous work for finding homomorphic lifts will suffice 
for showing that there is an automorphic lift.  The proof relies heavily on 
finite group theory, for which a good reference is \cite{Alge93}.
\begin{theorem}\label{sqrfree}
The automorphism $\varphi$ lifts to $\psi\in End(G)$ if and only if $\varphi$ 
lifts to some $\psi'\in Aut(G)$.
\end{theorem}
\begin{proof}
Let $\psi\in End(G)$ a lift of $\varphi$ be given.  Let $K=Ker(\psi)$ and 
$H=Im(\psi)$.  We will show that $G=K\times H$, from which the theorem follows
directly, since $\psi|_H$ is an isomorphism and $(Id,\psi|_H)$ 
will be a lift as desired, where $Id$ stands for the identity on $K$

From the proof of Lemma \ref{homtoaut}, $K\subseteq N$.  Since $\# N$ is 
squarefree and $N$ is abelian, it splits completely.  In particular, by the 
first isomorphism theorem, $N=K\times H_N$, where $H_N=Im(\psi|_N)$.  We know 
from the above proof that $HN=G$.  Then 
$$
G=H(K\times H_N)=(HK)(HH_N)=HKH=HK.
$$

Let $h\in H\cap K$ be given.  Notice first that $Im(\psi)\cap N = Im(\psi|_N)$
since the preimage of $N$ is contained in $N$.  Since $h\in K\subseteq N$, 
we have $h\in H_N$.  But $H_N\cap K=1$, so $h=1$.  Hence $G=H\ltimes K$.  
Therefore, since $K$ is central, $G=H\times K$.
\end{proof}

We are now going to construct a set of matrix equations whose solutions 
correspond to automorphic lifts of $\varphi$.  Let us first assume $N$ is 
cyclic and generated by $z$.  The key idea is that matching exponents of $z$ 
reduces to solving linear equations.  We first fix a noncommutative monomial
$$
f(\x):=\sprod{l=1}{s}x_{j_l}^{e_l}
$$
such that $f(\x)=z$.  Define 
$$
M_{m+1,j}:=\ssum{\s{l\in 1..s}{j_l=j}{}}{}e_l,
$$
which is the exponent of $x_j$ in the commutative image of $f(\x)$.  Since 
$\varphi(N)=N$, we can choose $w_{m+1}$ such that 
$z^{-w_{m+1}}=f(\overline{\x})$, where $\overline{x_i}$ is as above.

For an automorphism, the image of $z$ must be a generator of $N$.  We will 
define a matrix equation for each of the possible generators of $N$.  For $N$ 
infinite, define $w_{m+1}^{(1)}:=w_{m+1}+1$ and $w_{m+1}^{(2)}:=w_{m+1}-1$ and 
set
$$M':= 
\left(
\begin{array}{c}
M\\ 
M_{m+1}
\end{array}
\right).
$$
For $N$ finite, with $p$ the smallest prime dividing $\# N$, define 
$w_{m+1}^{(k)}:=w_{m+1}+k\frac{\# N}{p}$ for $k= 1,\dots, p-1$ and set 
$$M':= 
\left(
\begin{array}{c}
M\\ 
\frac{\#N}{p}M_{m+1}
\end{array}
\right).
$$
In either case, set 
$\w^{(k)}:=
\left(
\begin{array}{c}
\w\\
w_{m+1}^{(k)}
\end{array}
\right)
$
for every $k$.
\begin{theorem}\label{autmatrix}
Automorphic lifts of $\varphi\in Aut(G/N)$ are in one-to-one correspondence with 
solutions $\vv=(v_1,\dots,v_n)$ to the matrix equations 
$$
M'\vv=\w^{(k)}\pmod{ \#N}
$$
for $k=1,2$ if $N$ is infinite, and $k= 1,\dots, p-1$ for $N$ finite and $p$ 
the smallest prime dividing $\# N$.
\end{theorem}
\begin{proof}
For a solution $\vv$ to the equation $M\vv=\w$, we have a lift 
$\psi\in End(G)$ by Theorem \ref{matrix}.  By Lemma \ref{homtoaut}, 
$\psi\in Aut(G)$ if and only if $\psi(z)$ generates $N$.  But then 
\begin{eqnarray*}
\psi(z)&=&f(\psi(\x))=f(\overline{x_1}z^{v_1},\dots,\overline{x_n}z^{v_n})\\
&=&f(\overline{\x})f(z^{\vv})=z^{-w_{m+1}+\ssum{j=1}{n}m_{m+1,j} v_j}
\end{eqnarray*}
If $N$ is infinite we need $\psi(z)=z^{\pm 1}$ to generate $N$.  
This is equivalent to $-w_{m+1}+\ssum{j=1}{n}m_{m+1,j}v_j=\pm 1$.  These are 
exactly the matrix equations listed above.

If $N$ is finite, then we need $o(\psi(z))=\# N$, so we simply need $\psi(z)^{\frac{\# N}{p}}\neq 1$.  But this is equivalent to 
$$
-\frac{\#N}{p}w_{m+1}+\frac{\#N}{p}\ssum{j=1}{n}m_{m+1,j}v_j\neq 0 \pmod{\# N}.
$$
But, if this sum is not 0, it must be $k\frac{\#N}{p}$ for some 
$k= 1,\dots, p-1$.  These correspond to the above vectors.

\end{proof}

Similarly to above, we can generalize this to the case when $N$ is not cyclic.

If $N$ is generated by $z_1,\dots z_t$, then, following the above process, we 
can choose a presentation of each element $z_i$ and add another row to $M$ 
with this presentation, generating a matrix $M''$.  Next, we choose for each 
$z_i$ an element of $z_i'$ of $N$ which we would like to map $z_i$ to, such 
that the $z_i'$ also generate $N$.  Note that $z_i'$ must have the same order 
as $z_i$, since we have a homomorphism.

We write for each $k$ from 1 to the number of possible generator sets of 
$N$
$$z_i'=\sprod{j=1}{t} z_j^{{w_{i,m+j}}^{(k)}}.$$

\begin{corollary}
Automorphic lifts of $\varphi\in Aut(G/N)$ are in one-to-one correspondence with 
solutions to the matrix equations 
$$
\left(
\begin{smallmatrix}
M''&      &\\
 &\ddots&\\
 &      & M''\\
\end{smallmatrix}
\right)\vv = 
\left( 
\begin{smallmatrix}
\w_1^{(k)}\\
\vdots\\
\w_t^{(k)}
\end{smallmatrix}
\right)
\pmod{\#N},
$$
\end{corollary}

\begin{remark*}
Note that if we have 2 elements of infinite order as generators of $N$, then 
there are infinitely many choices for our $\w^{(k)}$, but otherwise we have a 
finite number of choices as above.
\end{remark*}

\section{An Application of the Above Techniques}
In this section, we give an application of the techniques developed above.  
Given a metacyclic group of order $p^n$ ($p$ an odd prime) represented by
$$
G:=\left<x,y|\ x^{p^{n-1}}, y^p, x^y=x^{1+p^{n-2}}\right>,
$$
we will show that $Inn(G)$ is not characteristic in $Aut(G)$.

To do this, we need information about the structure of $A:=Aut(G)$.  Schulte in \cite{MetaCyclic99} found the presentation for $A$

$$
A= \left< x_1,x_2,x_3 \left| 
\begin{smallmatrix} x_1^p,\ x_2^p,\ x_3^{(p-1)p^{n-2}},\
       x_1^{-a}\cdot  x_3^{-1}\cdot x_1 \cdot x_3\cdot  x_3^{j(p-1)p^{n-3}},\\ 
       x_2^{-a^{-1}}\cdot x_3^{-1}\cdot x_2\cdot x_3 \cdot x_3^{k(p-1)p^{n-3}},\
       x_1^{-1}\cdot x_2^{-1}\cdot x_1\cdot x_2\cdot x_3^{-(p-1)p^{n-3}}
\end{smallmatrix} \right.
\right>,
$$ 
where $a$ is a generator for the multiplicative group 
$(\integer/p^{n-2}\integer)\cross$, $a^{-1}$ is the multiplicative inverse 
$\pmod p$, and $j$ and $k$ are integers which can be determined explicitly.  

We shall henceforth refer to the center of $A$ as $Z:=Z(A)$.  By the 
commutator relations above $\left<x_3^{p-1}\right>\subseteq Z$.  Note that $A/\left<x_3^{p-1}\right>$
equals $(C_p\times C_p)\rtimes C_{p-1}$.  Since that group has trivial center, 
by the correspondence theorem, $\left<x_3^{p-1}\right>=Z$.  For ease of exposition, we 
will denote $x_3^{p-1}$ by $z$.  

Another important fact is that $I:=Inn(G)=\left<x_1,x_3^{(p-1)p^{n-3}}\right>$.  This 
follows from \cite{MetaCyclic99} by showing that conjugation by $x$ and $y$ 
correspond to the elements $x_1$ and $x_3^{(p-1)p^{n-3}}$, respectively.

From Section 1, we have a technique for determining when homomorphisms can be 
lifted from quotient groups.  Since $Z$ is large, cyclic, central, and 
characteristic in $A$, $A/Z$ is a natural candidate for this process.  

We will show that all elements of $Aut(A/Z)$ lift to $Aut(A)$, and then show
that some $\varphi \in Aut(A/Z)$ does not fix $IZ/Z$.  We then conclude 
that some $\psi \in Aut(A)$ does not fix $IZ$ and so $IZ$ is not 
characteristic in $A$.  Since $Z$ is characteristic, it follows that $I$ 
is not characteristic.

\begin{theorem}\label{inng}
The canonical mapping $\pi:Aut(A)\mapsto Aut(A/Z)$ is surjective.
\end{theorem}
\begin{proof}
Let $\varphi \in Aut(A/Z)$ be given and set $K:=\left<x_1,x_2\right>$.  
Studying $K$ in some detail will simplify the ensuing calculations.

Since $KZ$ is the unique Sylow $p$-subgroup of $A/Z$, it is characteristic, 
and hence $\varphi(KZ)=KZ$.  Therefore, without loss of generality, we can 
choose representatives $(\overline{x_1},\overline{x_2},\overline{x_3})$ such 
that $\varphi(x_iZ) = \overline{x_i}Z$ and 
$\overline{x_1},\overline{x_2}\in K$.  

Since $A/K$ is cyclic and hence abelian, it follows that $A'\subseteq K$.  
Notice that $K\cap Z=\left<[x_1,x_2]\right>=\left<z^{p^{n-3}}\right>$ and $K^p=1$.
Also, the exponent of $A$ is $(p-1)p^{n-2}$, so $A^{(p-1)p^{n-3}}$ contains 
only elements of order $p$, and hence is contained in $K$, since every 
element of order $p$ is in $K$.

Since $Z$ is cyclic and central, Theorem \ref{matrix} tells us that homomorphisms of $A$ that are lifts of $\varphi$ are in one-to-one correspondence with solutions to
$$
\left(
\begin{smallmatrix} 
p   & 0        & 0  \\ 
0   & p        & 0  \\ 
0   & 0        & (p-1)p^{n-2}\\
1-a & 0        & j(p-1)p^{n-3}\\ 
0   & 1-a^{-1} & k(p-1)p^{n-3}\\
0   & 0        & -(p-1)p^{n-3}
\end{smallmatrix} \right)
\left(
\begin{smallmatrix}
\\ \\  \vv \\ \\ \\  
\end{smallmatrix} \right) =
\left(
\begin{smallmatrix} 
\\ \\  \w \\ \\ \\ 
\end{smallmatrix}\right)
\pmod {\# Z}.
$$
So $\varphi$ lifts to a homomorphism if and only if this matrix is not degenerate.  We 
shall show that this matrix has $p^{n-3}$ solutions. 

Instead of calculating $\w$, we only need to show that $w_1, w_2, w_3=0$, 
and $p^{n-3}\mid w_4,w_5,w_6$.  Notice that the latter statement is equivalent 
to $z^{w_i}\in K$.

For $i=1,2$, we note that $K^p=1$ and for $i=3$, we have $r_3(\overline{\x})=1$
because the exponent of $A$ is $(p-1)p^{n-2}$, so any choice for 
$\overline{x_3}$ will give $w_3=0$.

For $i=4,5,6$, notice that since $\overline{x_1},\overline{x_2}\in K$, 
$$
r_i(\overline{\x})\in KA'A^{(p-1)p^{n-3}}.
$$  
From our comments above, we know that $A',A^{(p-1)p^{n-3}}\subseteq K$.  Hence 
$r_i(\overline{\x})\in K$, and we have $p^{n-3}\mid w_i$.

Since the matrix is taken mod $\# Z=p^{n-2}$, the third row of the matrix is 
all zeroes and $w_3=0$, so we can remove this redundancy.  Rows 1 and 2 
correspond to $pv_1=0$ and $pv_2=0$, respectively.  This is equivalent to 
$p^{n-3}\mid v_1,v_2$.

Therefore, if $\vv$ is a 
solution, it must be the case that $p^{n-3}\mid v_1,v_2$. Hence, we can remove 
the first three rows and write the matrix in the form:

$$
p^{n-3}\left(
\begin{smallmatrix} 
(1-a) & 0                & j(p-1)\\ 
0     &(1-a^{-1})        & k(p-1)\\
0     & 0                & -(p-1)
\end{smallmatrix}\right)
\left(\begin{smallmatrix}
\\ \vv \\ \\ \end{smallmatrix}\right)
=
p^{n-3}\left(\begin{smallmatrix} 
\\ \frac{\w}{p^{n-3}} \\ \\ \end{smallmatrix}\right)
\pmod{\# Z}
$$
With $v_1, v_2 \in \integer/p\integer$, and $v_3 \in \integer/p^{n-2}\integer$.

Solutions to this matrix will correspond to solutions of 
$$
\left(
\begin{smallmatrix}
1-a   &     0      &   -j  \\ 
 0    &  1-a^{-1}  &   -k  \\
 0    &     0      &  1  
\end{smallmatrix}
\right)
\left(
\begin{smallmatrix}
\\ \vv \\ \\ 
\end{smallmatrix}
\right) 
=
\left(
\begin{smallmatrix} 
\\ \frac{\w}{p^{n-3}} \\ \\ 
\end{smallmatrix}
\right)
\pmod{p}
$$

The determinant of this matrix is $D:=(1-a)(1-a^{-1})$, but  $a$ is a generator
for the multiplicative group $\left(\integer/(p^{n-1}\integer)\right)\cross$, 
so $a\neq 1 \pmod{ p}$ and $a^{-1}\neq 1 \pmod{ p}$.  Thus, $D$ is invertible, 
so the matrix is solvable.  Moreover, since $v_3\in \integer/p^{n-2}\integer$, 
and this solution only fixes $v_3 \pmod{p}$, we have $p^{n-3}$ choices for 
$v_3$, and hence $\varphi$ lifts to $p^{n-3}$ homomorphisms of $A$.  

We would now like to show that each of these homomorphisms is moreover an 
automorphism.  By Lemma \ref{homtoaut}, we know that $\psi$ is an 
automorphic lift exactly when $\psi(z)^{p^{n-3}}=[\psi(x_1),\psi(x_2)]\neq 1$.

\begin{lemma}
We have the equality 
$$
Z(K)=Z\cap K.
$$
\end{lemma}
\begin{proof}
Let $h=x_1^ax_2^b\left(z^{p^{n-3}}\right)^c\in Z(K)$ be given.  Then, since 
$[x_1,x_2]=z^{p^{n-3}}$, 
$$
hx_1=x_1h\left(z^{p^{n-3}}\right)^{-b}\text{ and }hx_2=x_2h\left(z^{p^{n-3}}\right)^a.
$$
Thus, $b=0=a$.
\end{proof}
We see that for $h\in K\wout Z(K)$, $C_K(h)=\left<h\right> Z(K)$, since 
$$
p^3=\#K>\#C_K(h)\geq p^2.
$$
But $\psi(x_1)\notin\left<\psi(x_2)\right>Z(K)$, as $\varphi(x_1Z)\notin\left<\varphi(x_2Z)\right>$.
Therefore, since $\psi(x_1)$ and $\psi(x_2)$ do not commute, 
$[\psi(x_1),\psi(x_2)]\neq 1$.  So the canonical mapping is surjective.  
\end{proof}

We will now proceed to show that $IZ/Z$ is not characteristic in $A/Z$.  Note 
that 
$$
A/Z=(IZ\times HZ)\rtimes \left<x_3\right>Z \iso \left(C_p\times C_p\right)
\rtimes C_{p-1},
$$
with the action of $x_3$ on $IZ$ being $x\mapsto x^a$ and the action of $x_3$ 
on $HZ$ being $x\mapsto x^{a^{-1}}$.  

We would like to show that 
$
(\overline{\x})=\left(x_2^{a^{-1}},x_1^a,x_3^{-1}\right)
$
extends to a homomorphism.  The presentation for $A/Z$ is the same as the 
presentation for $A$, adding the relation that $z=x_3^{p-1}=1$.  It is 
straightforward to calculate that $r_k(\overline{\x})=1$ for every $k$.  
Moreover, $\left<x_2^{a^{-1}},x_1^a,x_3^{-1}\right>=A/Z$, so this homomorphism is an 
automorphism.  This automorphism is the desired element of $Aut(A/Z)$.

\end{document}